\documentclass[a4paper,12pt]{article}
\usepackage{amsmath,amssymb,enumerate,amsthm,bbm}
\usepackage{color}
\usepackage{hyperref}
\hypersetup{
setpagesize=false,
 bookmarksnumbered=true,%
 bookmarksopen=true,%
 colorlinks=true,%
 linkcolor=blue,
 citecolor=red
}

\newcommand{\R}{\mathbb{R}}

\newcommand{\N}{\mathbb{N}}
\newcommand{\ep}{\varepsilon}
\newcommand{\pa}{\partial}

\DeclareMathOperator{\supp}{supp}
\newcommand{\lr}[1]{{}\langle{}#1{}\rangle{}}

\newcommand{\alternative}[1]{}
\newcommand{\renorm}[1]{{[\![}#1{]\!]}}
\newcommand{\reinner}[1]{{\langle\hspace{-3pt}\langle}#1{\rangle\hspace{-3pt}\rangle}}

\newtheorem{theorem}{Theorem}[section]
\newtheorem{lemma}[theorem]{Lemma}
\newtheorem{proposition}[theorem]{Proposition}
\newtheorem{corollary}[theorem]{Corollary}
\theoremstyle{remark}
\newtheorem{remark}{Remark}[section]
\theoremstyle{definition}

\newtheorem{definition}{Definition}[section]
\newtheorem{example}{Example}
\numberwithin{equation}{section}

\makeatletter
\def\@cite#1#2{[{{\bfseries #1}\if@tempswa , #2\fi}]}
\makeatother                  

\begin{document}

\begin{center}
\Large{{\bf
Remarks on criticality theory for Schr\"odinger operators and its application to wave equations with potentials
}}
\end{center}

\vspace{5pt}

\begin{center}
Motohiro Sobajima%
\footnote{
Department of Mathematics, 
Faculty of Science and Technology, Tokyo University of Science,  
2641 Yamazaki, Noda-shi, Chiba, 278-8510, Japan,  
E-mail:\ {\tt msobajima1984@gmail.com}}
\end{center}

\newenvironment{summary}{\vspace{.5\baselineskip}\begin{list}{}{%
     \setlength{\baselineskip}{0.85\baselineskip}
     \setlength{\topsep}{0pt}
     \setlength{\leftmargin}{12mm}
     \setlength{\rightmargin}{12mm}
     \setlength{\listparindent}{0mm}
     \setlength{\itemindent}{\listparindent}
     \setlength{\parsep}{0pt}
     \item\relax}}{\end{list}\vspace{.5\baselineskip}}
\begin{summary}
{\footnotesize {\bf Abstract.}
In this paper, we give an alternative perspective of the criticality theory for (nonnegative) Schr\"odinger operators. Schr\"odinger operator $S$ is classified as subcritical/critical in terms of the existence/nonexistence of a positive Green function for the associated elliptic equation $Su=f$. Such a property strongly affects to the large-time behavior of solutions to the parabolic equation $\pa_tv+Sv=0$. In this paper, we propose a remarkable quantity in terms of the structure of Hilbert lattices, which keeps some important properties including the notion of criticality theory. As an application, we study the large-time behavior of solutions to the hyperbolic equation $\pa_t^2w+Sw=0$.

}
\end{summary}

\tableofcontents

\section{Introduction}

We consider $N$-dimensional Schr\"odinger operators in $\Omega$ 
of the form 
\begin{equation}\label{intro:Schr}
S=-\Delta + V\quad\text{in}\ \Omega, 
\end{equation}
where $\Delta=\sum_{k=1}^N\frac{\pa^2}{\pa x_k^2}$ is the 
usual $N$-dimensional Laplacian, 
$\Omega$ is a connected domain 
and $V$ is real-valued continuous function on $\Omega$. 
Typical examples of the choice of $\Omega$
include the whole space $\R^N$ $(N\geq1)$, 
exterior domains in $\R^N$ $(N\geq 2)$ and also 
the punctured space $\Omega=\R^N\setminus\{0\}$ ($N\geq 2$).
Throughout this paper, we assume that 
the minimal realization $S_{\min}=S|_{C_0^\infty(\Omega)}$ of $S$ is nonnegative in $L^2(\Omega)$,  that is,
 \begin{equation}\label{condition:nonnegative}
\int_{\Omega}\Big(|\nabla u(x)|^2+V(x)|u(x)|^2\Big)\,dx\geq 0, \quad u\in C_0^\infty(\Omega).
\end{equation}
In Simon \cite{Simon1981}, 
the notion of subcriticality and criticality initially appears
to classify nonnegative Schr\"odinger operators 
via the long-time behavior of 
the associated heat kernel $p_S$. 
Here the heat kernel $p_S$ 
is the integral kernel of 
the solution operator 
\begin{align*}
f\mapsto v(\cdot,t)=\int_{\Omega}p_S(\cdot,y,t)f(y)\,dy
\end{align*}
of the initial-value problem:
\begin{equation}\label{intro:eq:para}
\begin{cases}
\pa_tv(x,t)+Sv(x,t)=0 &\text{in}\ \Omega\times (0,\infty),
\\
v(x,t)=0 &\text{on}\ \pa\Omega\times (0,\infty),  
\\
v(x,0)=f(x) &\text{in}\ \Omega. 
\end{cases}
\end{equation}
One important
characterization of subcriticality (resp.\ criticality) 
is described in terms of the existence (resp.\ nonexistence) of the positive minimal Green function $\Phi_S$ 
for the associated elliptic equation $Su=f$ (see Murata \cite{Murata1986}).
The above different descriptions are connected to 
\begin{equation}\label{intro:kernel-green}
\Phi_S(x,y)=\int_{0}^\infty p_S(x,y,t)\,dt
\begin{cases}
\in L_{\rm loc}^1(\Omega\times \Omega)\ \text{if $S$ is subcritical},
\\
=+\infty\ \text{if $S$ is critical}.
\end{cases}
\end{equation}
From the stochastic point of view, 
subcriticality (resp.\ criticality) of 
$-\Delta$ in $\R^N$ corresponds to the transience (resp.\ recurrence) 
of the $N$-dimensional Brownian motion (see Durrett \cite{Book_Durrett}). 
We shall start discussing the notion of the ciriticality theory from 
the following characterization 
in terms of the negative perturbation of Schr\"odinger operators (see also Murata \cite{Murata1986}).

\begin{definition}\label{def:criticailty-Schroedinger}
Let $V$ satisfy \eqref{condition:nonnegative}.
We define 
\begin{itemize}
\item[\bf (i)] $S$ is {\it subcritical} if for every nonnegative function 
$W\in C_0^\infty(\Omega)$, there exists a positive constant $\delta$ such that 
\[
\delta\int_{\Omega}W(x)|u(x)|^2\,dx
\leq 
\int_{\Omega}\Big(|\nabla u(x)|^2+V(x)|u(x)|^2\Big)\,dx, \quad u\in C_0^\infty(\Omega).
\]
\item[\bf (ii)]
$S$ is {\it critical} if $S$ is not subcritical.
\end{itemize}
\end{definition}
\begin{remark}
If \eqref{condition:nonnegative} is false, then $S$ is called supercritical, but 
this situation does not appear in this paper.
\end{remark}

The aim of this paper is to propose some alternative interpretation 
for subcriticality/criticality for Schr\"odinger operators.
\begin{example}\label{exam:hardy-potential}
In the simplest case $S_0=-\Delta_{\R^N}$ ($\Omega=\R^N$ and $V\equiv0$), 
the situation is the following. 
If $N\geq 3$, then the well-known Hardy inequality
\[
C_{\rm H}
\int_{\R^N}\frac{|u(x)|^2}{|x|^2}\,dx
\leq 
\int_{\R^N}|\nabla u(x)|^2\,dx, \quad u\in H^1(\R^N)
\]
with the optimal constant $C_{\rm H}=(\frac{N-2}{2})^2$ 
directly gives that $S_0$ is subcritical. 
In contrast, one can see 
by a scaling argument 
that $S_0$ is critical if $N\leq 2$. 
In this connection, we can also consider 
Schr\"odinger operators with inverse-square potentials
\[
S_{\lambda}=-\Delta + \frac{\lambda }{|x|^2}, \quad \text{in} \ \Omega=\R^N\setminus \{0\}\ (N\geq 2),
\]
where $\lambda\geq \lambda_*=-C_{\rm H}$ is assumed to ensure the nonnegativity of $S_\lambda$. 
In this case, 
$S_\lambda$ is subcritical if $\lambda>\lambda_*$
and $S_{\lambda_*}$ is critical.
\end{example}

By using the standard method of sesquilinear forms 
in Hilbert spaces
(see e.g., Kato \cite{Book_Kato}
and Ouhabaz \cite{Book_Ouhabaz}), 
it is not difficult to prove that 
there uniquely exists the Friedrichs extension of the minimal realization $S_{\min}$,  
which successfully describes the Dirichlet boundary condition at 
the ``regular'' point on $\pa\Omega$ (if $\pa\Omega\neq \emptyset$). Hereafter, we use the notation $S$ as the corresponding Friedrichs extension of $S_{\min}$. 
The positivity preserving property of the generated semigroup $\{e^{-tS}\}_{t\geq0}$ 
(in $L^2(\Omega)$) can be also verified by using 
the positive ($H_{\rm loc}^{2}$-)function $\psi$ satisfying $S\psi=0$, 
which exists under the condition \eqref{condition:nonnegative}
(see e.g., Agmon \cite{Agmon1982}). 

\begin{remark}
In this connection, we can also generalize the class of potentials. 
For instance, we can consider $V=V_+-V_-$, where 
$V_+$ and belong to 
the local Kato class and the Kato class, 
respectively (for detail, see e.g., Lucia--Prashanth \cite{LP2018}). 
But we do not enter the detail to avoid an awkward consideration.
\end{remark}

In the present paper, 
we propose an alternative interpretation of subcriticality and criticality. 
The following assertion gives a remarkable perspective 
involving an equivalent statement of subcriticality. 
To describe this, we use 
$R(A)=\{Au\;;\;u\in D(A)\}$ (the range of the operator $A$) 
and 
$[L^2(\Omega)]_+=\{f\in L^2(\Omega)\;;\;f\geq 0\}$.

\begin{theorem}\label{thm:equivalence}
Let $S$ be the Friedrichs extension of $S_{\min}$ satisfying \eqref{condition:nonnegative}
and let $\alpha>0$. The following assertions are equivalent:
\begin{itemize}
\item[\bf (i)] 
$R(S^{\alpha})\cap [L^2(\Omega)]_{+}\neq \{0\}$.
\item[\bf (ii)] 
$C_0^\infty(\Omega)\subset R(S^{\alpha})$.
\end{itemize}
In particular,
the above two conditions with $\alpha=\frac{1}{2}$ 
hold if and only if $S$ is subcritical.
\end{theorem}
We emphasize that 
Theorem \ref{thm:equivalence} suggests that 
the interval
\[
I_S=\big\{\alpha> 0\;;\;R(S^{\alpha})\cap [L^2(\Omega)]_{+}\neq \{0\}\big\}
\]
has precise information including subcriticality and criticality. 
Moreover, 
since the interval $I_S$ can be 
defined by using only functional analytic tools ($C_0$-semigroups and positive cones),
this representation enables us to extend the criticality theory for 
Schr\"odinger operators to that for nonnegative 
selfadjoint operators in Hilbert lattices.

\begin{example}
In the case
$S_0=-\Delta_{\R^N}$, 
we can determine $I_{S_0}=(0,\frac{N}{4})$ 
and therefore one has a kind 
of characterization of the dimension $N=4\sup I_{S_0}$.
A similar consideration is also applicable to 
the sub-Laplacian $\mathcal{L}$ on the nilpotent Lie group $G$ 
with the homogeneous dimension $Q>0$ which was effectively used 
with another aspect (see e.g., Folland \cite{Folland1975}).
\end{example}
\begin{example}
We can also consider the (Dirichlet) Laplacian $\Delta_\Omega$
in an exterior domain $\Omega\subset \R^N$.
Then $I_{-\Delta_{\Omega}}$ can be determined as
$I_{-\Delta_{\Omega}}=(0,\frac{1}{2}]$ if $N=2$ and  
$I_{-\Delta_{\Omega}}=(0,\frac{N}{4})$ if $N\geq 3$. 
Especially, for the case $N=2$, it can be shown 
by the heat kernel estimate in Grigoryan and Saloff-Coste \cite{GS2002}. 
\end{example}

Moreover, if $S$ does not have zero-eigenvalue, 
then the range $R(S^{\alpha})$ naturally equips the norm 
\[
\|g\|_{R(S^{\alpha})}=
\|g\|_{L^2(\Omega)}
+
\|g_*\|_{L^2(\Omega)} \quad(g_*\in D(S^\alpha)\text{ and }S^{\alpha}g_*=g)
\]
and forms a Banach space. 
If $\alpha\in I_S$, then $C_0^\infty(\Omega)$ 
can be used as a dense subset of $R(S^{\alpha})$ as usual. Namely, 
\begin{theorem}\label{thm:density}
Let $S$ be the Friedrichs extension of $S_{\min}$ satisfying \eqref{condition:nonnegative}. 
If $\alpha\in I_S$, then 
$C_0^\infty(\Omega)$ is dense in $R(S^\alpha)$ with respect to the topology induced by the norm 
$\|\cdot\|_{R(S^{\alpha})}$. 
\end{theorem}

From the observation with $I_S$ as above, 
we shall consider the following 
initial-boundary value problem of 
linear wave equations with potentials of the form
\begin{equation}\label{intro:eq:hyper}
\begin{cases}
\pa_t^2w(x,t)-\Delta w(x,t)+V(x)w(x,t)=0 &\text{in}\ \Omega\times (0,\infty),
\\
w(x,t)=0&\text{on}\ \pa\Omega\times(0,\infty),
\\
(w,\pa_tw)(x,0)=(u_0(x),u_1(x)) &\text{in}\ \Omega. 
\end{cases}
\end{equation}
For the classical wave equation ($V\equiv 0$) in $\R^N$, we refer the book of Racke \cite{Book_Racke}.
We recall that $\Omega$ is connected and $V\in C(\Omega)$ satisfies \eqref{condition:nonnegative}.  
We use the notation $W_S(\cdot)g$ as the solution $u(\cdot)$ of \eqref{intro:eq:hyper} with the initial data $(u_0,u_1)=(0,g)$. Since the energy conservation law
\[
\|\pa_tw(t)\|_{L^2(\Omega)}^2+\|S^{1/2}w(t)\|_{L^2(\Omega)}^2
= \|w_1\|_{L^2(\Omega)}^2+\|S^{1/2}w_0\|_{L^2(\Omega)}^2, \quad t\geq 0
\]
holds, the $L^2$-bound $\|W_S(t)g\|_{L^2(\Omega)}\leq t\|g\|_{L^2(\Omega)}$ $(t\geq 0)$
is valid in general.

The following theorem describes 
a typical property of solutions to 
\eqref{intro:eq:hyper} via the interval $I_S$. 
We emphasize that a kind of the large-time behavior
of solutions to the hyperbolic problem \eqref{intro:eq:hyper} 
can be observed by using a several estimate for the parabolic problem \eqref{intro:eq:para}.

\begin{theorem}\label{thm:wave-bounded}
Let $S$ be the Friedrichs extention of $S_{\min}$ satisfying \eqref{condition:nonnegative}. 
If $\alpha\in I_S\cap (0,\frac{1}{2}]$, then for every 
$g\in C_0^\infty(\Omega)$, 
the solution $W_S(\cdot)g$ satisfies $t^{2\alpha-1}W_S(\cdot)g\in L^\infty(0,\infty;L^2(\Omega))$.
\end{theorem}

\begin{remark}\label{rem:wave-bound}
In the case $S_0=-\Delta_{\R^N}$, 
it is not difficult to deduce 
the following estimate 
via the representation of solutions 
$\mathcal{F}[W_{S_0}(t)g](\xi)=\frac{\sin(t|\xi|)}{|\xi|}\mathcal{F}g(\xi)$
($\mathcal{F}$ is the Fourier transform)
and the 
Hardy--Littlewood-Sobolev inequality:
for every $q\in (1,2)\cap [\frac{2N}{N+2},2)$, 
then there exists a positive constant $C$ such that 
for every $g\in L^2(\R^N)\cap L^{q}(\R^N)$, 
\[
t^{\frac{N}{q}-\frac{N}{2}-1} 
\|W_{S_0}(t)g\|_{L^2(\R^N)}
\leq 
C\|g\|_{L^q(\R^N)}, \quad t>0.
\]
\end{remark}

The following assertion for 
subcritical Schr\"odinger operators is 
deduced as a corollary of Theorem \ref{thm:wave-bounded}. 
\begin{corollary}\label{cor:subcritical--bounded}
Let $S$ be the Friedrichs extention of $S_{\min}$ satisfying \eqref{condition:nonnegative}. 
If $S$ is subcritical, 
then for every 
$g\in C_0^\infty(\Omega)$, 
the solution $W_S(\cdot)g$ is uniformly bounded in $L^2(\Omega)$.
\end{corollary}

\begin{remark}
In Morawetz \cite{Morawetz}, 
she had provided a way to estimate the $L^2$-norm of solutions 
to the three-dimensional wave equation in exterior domain 
by solving the Poisson equation involving initial data. 
The application of the interval $I_S$ 
to the wave equation in this paper 
may be regarded as one of generalization of her mothod mentioned above. 
\end{remark}
\begin{remark}
It is mentioned in Ikehata \cite{Ikehata2023} that 
if $g\in L^2(\R^N)\cap L^1(\R^N)$, then 
the solution $W_{S_0}(\cdot)g$ ($S_0=-\Delta_{\R^N}$) of 
the classical wave equation in $\R^N$ 
satisfies 
\[
\|W_{S_0}(t)g\|_{L^2(\R^N)}
\lesssim 
\begin{cases}
(1+t)^{1/2}
&\text{if}\ N=1\ 
\text{and}\ \int_{\R}g\,dx\neq 0,
\\
(1+\log (1+t))^{1/2}
&\text{if}\ N=2\ 
\text{and}\ \int_{\R^2}g\,dx\neq 0,
\\
1&\text{if}\ N\geq 3\ \text{or}\ \int_{\R^N}g\,dx=0
\end{cases} 
\]
and the above estimates are optimal. 
Boundedness of solutions to one-dimensional wave equations with (nontrivial) nonnegative potentials 
was also discussed in Ikehata \cite{Ikehata2024}.  
The same conclution for 
two-dimensional wave equations
with (nontrivial) nonnegative potentials
can be also shown by Corollary \ref{cor:subcritical--bounded}. 
Although precise asymptotics of solutions to \eqref{intro:eq:hyper} 
is not treated in Theorem \ref{thm:wave-bounded}, 
we could observe a general structure of wave equations with potentials through the interval $I_S$. 
\end{remark}

We also give a comment on the validity of the opposite assertion of Theorem \ref{thm:wave-bounded}. 
\begin{theorem}\label{thm:opposite}
Let $S$ be the Friedrichs extention of $S_{\min}$ satisfying \eqref{condition:nonnegative}. 
Assume that 
there exists a constant 
$\alpha_*\in (0,\frac{1}{2}]$ such that 
for every $g\in C_0^\infty(\Omega)$, 
the solution $W_S(\cdot)g$ 
satisfies $t^{2\alpha_*-1}W_S(\cdot)g\in L^\infty(0,\infty;L^2(\Omega))$.
Then $(0,\alpha_*)\subset I_S$. 
\end{theorem}
\begin{remark}
The one-dimensional operator $S_0=-\frac{d^2}{dx^2}$ in $\R$ 
satisfies $I_{S_0}=(0,\frac{1}{4})$. 
We can also directly calculate $t^{-\frac{1}{2}}W_{S_0}(\cdot)g\in L^\infty(0,\infty;L^2(\Omega))$ for every 
$g\in C_0^\infty(\R)$ by using the d'Alembert formula.
This is the case where $t^{2\alpha-1}W_{S_0}(\cdot)g\in L^\infty(0,\infty;L^2(\R))$ is valid 
with $\alpha=\frac{1}{4}$ but $\frac{1}{4}\notin I_{S_0}$.

We do not know whether there exists 
a critical Schr\"odinger operator $S$ (that is, $\frac{1}{2}\notin I_S$)
such that 
all solutions $W_S(\cdot)g$ ($g\in C_0^\infty(\Omega)$)
of the corresponding wave equation \eqref{intro:eq:hyper}
are bounded;
in this case we must have $I_S=(0,\frac{1}{2})$. 
\end{remark}

We also give a comment on Schr\"odinger operators with inverse-square potentials. 
We consider the case of $S=-\Delta+\frac{\lambda}{|x|^2}$ in 
the punctured space
$\Omega=\R^N\setminus \{0\}$ ($N\geq2$). Namely, 
we discuss solutions of wave equations with inverse-square potentials:
\begin{equation}\label{intro:wave-hardy-potential}
\begin{cases}
\pa_t^2w(x,t)-\Delta w(x,t)+\dfrac{\lambda}{|x|^2}w(x,t)=0, & \text{in}\ (\R^N\setminus\{0\})\times (0,\infty), 
\quad t>0
\\
(u,u')(x,0)=(0,g(x)),
\end{cases}
\end{equation}
where $\lambda\geq\lambda_*$ (see Example \ref{exam:hardy-potential}). 
It should be noticed that selfadjoint extensions of $S_{\min}$ 
are not unique when $\lambda<\lambda_*+1$ 
but we can define the Friedrichs extension $S_\lambda$. This means that 
we implicitly impose a boundary condition at the singular point $x=0$.
As a consequence of Corollary \ref{cor:subcritical--bounded}, we have 
\begin{corollary}\label{cor:hardy-subcritical}
Let $N\geq 2$, $\lambda >\lambda_*$ 
and let $S_\lambda$ be the Friedrichs extension of $-\Delta+\frac{\lambda}{|x|^2}$ 
in $L^2(\R^N)$.
Then for every $g\in C_0^\infty(\R^N\setminus\{0\})$, 
the solution $W_{S_\lambda}(\cdot)g$ is uniformly bounded in $L^2(\R^N)$.  
\end{corollary}

Concerning the threshold case $\lambda=\lambda_*$, 
we recall that $S_{\lambda_*}$ is critical and therefore 
Corollary \ref{cor:subcritical--bounded} is not applicable. 
In this case, we can observe that 
the solution $W_{S_{\lambda_*}}(\cdot)g$ is not bounded in $L^2(\Omega)$. 
The case $N=2$ is already dealt with in \cite{Ikehata2023}. 
The same conclusion also holds for the case $N\geq 3$. 
\begin{proposition}\label{prop:hardy-critical}
Let $N\geq 3$ and let $S_{\lambda_*}$ be the Friedrichs extension of $-\Delta+\frac{\lambda_*}{|x|^2}$ in $L^2(\R^N)$ with $\lambda_*=-C_{\rm H}$. 
If $g\in C_0^\infty(\R^N\setminus\{0\})$ is
radially symmetric 
and satisfies
\begin{equation}\label{intro:prop:hardy-critical:ass}
\int_{\R^N}g(x)|x|^{-\frac{N-2}{2}}\,dx\neq 0,
\end{equation}
then 
there exists a constant $C\geq 1$ such that 
\[
C^{-1} (\log t)^{1/2}
\leq \|W_{S_{\lambda_*}}(t)g\|_{L^2(\R^N)}
\leq C(\log t)^{1/2}, \quad t\geq 2.
\]
\end{proposition}

The present paper is organized as follows:
In Section \ref{sec:subcrit}, we give our perspective 
by proving 
Theorems \ref{thm:equivalence} and \ref{thm:density}. 
A general structure of the range of fractional operators of nonnegative selfadjoint 
operators is explained in Section \ref{sec:general}. 
An upper bound for solutions to second-order evolution equations (related to wave equations)
is also stated in Section \ref{sec:general}. 
Finally, in Section \ref{sec:app},
we prove the series of assertions 
for wave equations with potentials.

\section{Relation between Subcriticality and the range $R(S^{\alpha})$}
\label{sec:subcrit}

In this section, we prove Theorems 
\ref{thm:equivalence} and \ref{thm:density}.
We first note the important description of subcriticality
(see Murata \cite{Murata1986}).
\begin{lemma}\label{lem:Green}
Let $S$ be the Friedrichs extension of  $S_{\min}$ satisfying \eqref{condition:nonnegative}. 
The following assertions are equivalent:
\begin{itemize}
\item[\bf (i)] $S$ is subcritical. 
\item[\bf (ii)] There exists a positive function 
$\Phi\in L_{\rm loc}^1(\Omega\times \Omega)$ such that 
for every $f\in C_0^\infty(\Omega)$, the function 
\[
u(x)=\int_{\Omega}\Phi(x,y)f(y)\,dy
\]
satisfies $\Delta u=Vu-f$ $(Su=f)$ in the classical sense. 
\end{itemize}
In this case, the function
\begin{equation}\label{eq:kernel-green}
\Phi_S(x,y)=\int_{0}^\infty p_S(x,y,t)\,dt
\end{equation}
is well-defined in $L_{\rm loc}^1(\Omega\times\Omega)$ 
and is the smallest function among $\Phi$ satisfying the property {\bf (ii)}.
\end{lemma}

The following lemma plays a central role in the discussion of the present paper. 
We will give a proof later (in Section \ref{sec:general} with a general setting). 
\begin{lemma}\label{lem:characterization-range}
Let $S$ be the Friedrichs extension of  $S_{\min}$ satisfying \eqref{condition:nonnegative}. 
Then for every $\alpha>0$, 
the range $R(S^\alpha)$ 
has the following representation with the $C_0$-semigroup 
$\{e^{-tS}\}_{t\geq0}$ in $L^2(\Omega)$:
\[
R(S^{\alpha})=
\left\{g\in L^2(\Omega)\;;\;\renorm{g}_{R(S^\alpha)}<+\infty\right\}, 
\]
where
\begin{equation}\label{eq:norm-range}
\renorm{g}_{R(S^\alpha)}=
\left(\int_{0}^\infty
\|t^{\alpha} e^{-tS}g\|_{L^2(\Omega)}^2\,\frac{dt}{t}\right)^{1/2}.
\end{equation}
\end{lemma}

\begin{proof}[Proof of Theorem \ref{thm:equivalence}: the part of {\bf(i)}$\iff${\bf(ii)}]
The assertion {\bf (ii)}$\implies${\bf (i)} is trivial.  
We now prove we treat the opposite direction {\bf (i)}$\implies${\bf (ii)}. 
Fix $g_0\in R(S^{\alpha})\cap [L^2(\Omega)]_{+}$ with $g_0\neq 0$.
In view of Lemma \ref{lem:characterization-range}, we have 
\[
\renorm{g_0}_{R(S^\alpha)}^2=\int_{0}^\infty\|t^{\alpha}e^{-t S}g_0\|_{L^2(\Omega)}^2\,\frac{dt}{t}<+\infty. 
\]
Let $g\in C_0^\infty(\Omega)$ be arbitrary fixed. 
Since the heat kernel $p_S$ is continuous and 
positive everywhere and $g$ is compactly supported, 
there exists positive constants $t_0$ and $C$ such that 
$|g|\leq Ce^{-t_0S}g_0$ on $\Omega$. 
Therefore we have
$|e^{-t S}g|\leq e^{-t S}|g|
\leq C e^{-(t+t_0)S}g_0$ for all $t\geq 0$.
This together with the $L^2$-contraction for $\{e^{-tS}\}_{t\geq0}$ shows that 
\begin{align*}
\int_0^\infty\|t^\alpha e^{-tS}g\|_{L^2(\Omega)}^2\,\frac{dt}{t}
&\leq 
C^2\int_0^\infty\|t^\alpha e^{-(t+t_0) S}g_0\|_{L^2(\Omega)}^2\,\frac{dt}{t}
\\
&\leq 
C^2\int_0^\infty\|t^\alpha e^{-t S}g_0\|_{L^2(\Omega)}^2\,\frac{dt}{t}
\\
&
=
C^2\renorm{g_0}_{R(S^\alpha)}^2.
\end{align*}
Using Lemma \ref{lem:characterization-range} again, we see $g\in R(S^{\alpha})$.
\end{proof}

The following lemma is a direct relationship 
between $\frac{1}{2}\in I_S$ and the existence of 
Green function $\Phi_S$ for $Su=f$.
\begin{lemma}\label{lem:DtoGreen}
Let $S$ be the Friedrichs extension of  $S_{\min}$ satisfying \eqref{condition:nonnegative}. 
Then the following assertions are equivalent:
\begin{itemize}
\item[\bf (i)] $C_0^\infty(\Omega)\subset R(S^{\alpha})$.
\item[\bf (ii)]The function 
\[
\Phi_{S,\alpha}(x,y)=\int_0^\infty t^{2\alpha}p_S(x,y,t)\,\frac{dt}{t}
\]
converges in $L_{\rm loc}^1(\Omega\times\Omega)$, 
where $p_S$ is the associated heat kernel of $\{e^{-tS}\}_{t\geq0}$.
\end{itemize}
\end{lemma}
\begin{proof}
We show {\bf (i)}$\implies${\bf (ii)}.
For arbitrary fixed compact set $K\subset \Omega$, 
we take $\psi\in C_0^\infty(\Omega)$ satisfying 
$\psi\geq \mathbbm{1}_K$, 
where $\mathbbm{1}_A$ is the indicator function of $A$. 
Then using the associated heat kernel $p_S$, we have
\begin{align*}
\int_{0}^n \|t^{\alpha}e^{-tS}\psi\|_{L^2(\Omega)}^2\,\frac{dt}{t}
&=
\int_{0}^n t^{2\alpha-1}\lr{e^{-2tS}\psi,\psi}_{L^2(\Omega)}\,dt
\\
&=
\iint_{\supp \psi \times \supp \psi}\left(\int_{0}^n t^{2\alpha-1}p_S(x,y,2t)\,dt\right)\psi(x)\psi(y)\,dx\,dy.
\\
&\geq 
4^{-\alpha}\iint_{K\times K}\left(\int_{0}^{2n} \tau^{2\alpha}p_S(x,y,\tau)\,\frac{d\tau}{\tau}\right)\,dx\,dy.
\end{align*}
The monotone convergence theorem implies that 
$\|\Phi_{S,\alpha}\|_{L^1(K\times K)}\leq 4^{\alpha}\renorm{\psi}_{R(S^{\alpha})}^2$
which is finite by Lemma \ref{lem:characterization-range}.
Since $K$ is chosen arbitrary, we can deduce the same conclusion 
for arbitrary compact set in $\Omega\times \Omega$. 

Next we prove {\bf (ii)}$\implies${\bf (i)}.
Since the assertion of the equivalence in Theorem \ref{thm:equivalence} has been proved, 
it suffices to show that $\mathbbm{1}_{K_0}\in R(S^{\alpha})$
for some compact set $K_0(\neq \emptyset)$ in $\Omega$. 
Fix $n\in\N$ arbitrary. Similarly as above, we have 
\begin{align*}
\int_{0}^n \|t^{\alpha}e^{-tS}\mathbbm{1}_{K_0}\|_{L^2(\Omega)}^2\,\frac{dt}{t}
&=
4^{-\alpha}\iint_{K_0\times K_0}\left(\int_{0}^{2n} t^{2\alpha}p_S(x,y,t)\,\frac{dt}{t}\right)\,dx\,dy
\\
&	\leq 4^{-\alpha}\|\Phi_{S,\alpha}\|_{L^1(K_0\times K_0)}.
\end{align*}
This shows that $\renorm{\mathbbm{1}_{K_0}}_{R(S^{\alpha})}<+\infty$. 
As a consequence, Lemma \ref{lem:characterization-range} yields $\mathbbm{1}_{K_0}\in R(S^\alpha)$.    
\end{proof}

\begin{proof}[Proof of Theorem \ref{thm:equivalence} (the rest)]
By Lemmas \ref{lem:Green} 
and \ref{lem:DtoGreen}, 
we have that 
$S$ is subcritical if and only if 
$C_0^\infty(\Omega)\subset R(S^{1/2})$, but  
it has found out that  such a condition is equivalent to $\frac{1}{2}\in I_S$.  
The proof is complete.
\end{proof}
\begin{proof}[Proof of Theorem \ref{thm:density}]
Let $f\in R(S^\alpha)$ be nonnegative. Then using the monotone convergence theorem, we have 
\begin{align*}
\int_{\Omega\times \Omega}\left(\int_0^n t^{2\alpha-1}p_S(x,y,t)\,dt\right)f(x)f(y)\,dx\,dy
&=
\int_0^n \|t^\alpha e^{-\frac{t}{2}}f\|_{L^2(\Omega)}^2\,\frac{dt}{t}
\\
&
\to 4^{\alpha}
\renorm{f}_{R(S^\alpha)}^2
\end{align*}
as $n\to \infty$. 
The dominated convergence theorem provides that 
the following integral representation is valid:
\[
\renorm{f}_{R(S^\alpha)}^2
=4^{-\alpha}\iint_{\Omega\times \Omega}\Phi_{S,\alpha}(x,y)f(x)f(y)\,dx\,dy.
\]
Thanks to the above representation, we can see 
from the monotone convergence theorem again that the sequence $\{\mathbbm{1}_{\Omega_n}f\}_{n\in\N}$ 
(with $\Omega_n=\{x\in\Omega\;;\;|x|<n, {\rm dist}(x,\pa\Omega^c)>n^{-1}\}$)
converges to $f$ in the sense of the topology 
induced by the norm $\|\cdot\|_{R(S^{\alpha})}
=\|\cdot\|_{L^2(\Omega)}+\renorm{\cdot}_{R(S^\alpha)}$. Combining 
the regularization with the Friedrichs mollifier, 
we can obtain the density of $C_0^\infty(\Omega)$. 
\end{proof}

\section{A general observation for $R(S^{\alpha})$}\label{sec:general}

Let $X$ be a complex Hilbert space 
endowed with inner product 
$\lr{\cdot,\cdot}_X$ and norm $\|\cdot\|_X$. 
We fix a nonnegative selfadjoint operator $A$ in $X$. 
To simplify the situation, 
throughout this section we always assume $0\in \sigma(A)\setminus\sigma_{\rm p}(A)$.
\begin{remark}
In the case $0\notin \sigma(A)$ (or equivalently $0\in \rho(A)$), we immediately have $R(A)=X$.
On the one hand, 
if $0\in \sigma_{\rm p}(A)$, then $R(A^{\alpha})$ is always orthogonal to the null space 
$N(A)$ $(\neq \{0\})$. 
\end{remark}

The following assertion is one of representation for the range 
$R(A^\alpha)$ for $\alpha>0$; 
note that the assertion includes Lemma \ref{lem:characterization-range}.
\begin{lemma}
\label{lem:range:complete}
For every $\alpha>0$, the range $R(A^\alpha)$ 
has the following representation:
\[
R(A^\alpha)=
\left\{x\in X\;;\;\renorm{x}_{R(A^\alpha)}<+\infty\right\}, 
\quad
\renorm{x}_{R(A^\alpha)}=
\left(
\int_{0}^\infty
\|t^{\alpha}e^{-tA}x\|_X^2
\,\frac{dt}{t}\right)^{1/2}.
\]
In particular, the vector space $R(A^{\alpha})$ endowed with the inner product 
\[
\lr{x,y}_{R(A^{\alpha})}
=
\lr{x,y}_{X}
+
\reinner{x,y}_{R(A^{\alpha})}, 
\quad 
\reinner{x,y}_{R(A^{\alpha})}
=\int_{0}^\infty t^{2\alpha}\lr{e^{-tA}x,e^{-tA}y}_X\,\frac{dt}{t}
\]
forms a Hilbert space. 
\end{lemma}

\begin{proof}
Assume that $\renorm{x}_{R(A^\alpha)}<+\infty$.
We define the sequence
\[
y_n=
\frac{1}{\Gamma(\alpha)}\int_0^n t^{\alpha}e^{-tA}x\,\frac{dt}{t}, \quad n\in\N.
\]
Then for every $n,m\in \N$ satisfying $m\geq n$, one can compute 
\begin{align*}
\Gamma(\alpha)^2\|y_n-y_m\|_{X}^2
&=
\int_n^m\int_n^m t^{\alpha-1}s^{\alpha-1}\|e^{-\frac{t+s}{2}A}x\|_X^2\,dt\,ds 
\\
&\leq 
4^{\alpha}B(\alpha,\alpha)\int_n^m t^{2\alpha-1}\|e^{-\tau A}x\|_X^2\,d\tau,
\end{align*}
where we have used the change of variables $(t,s)=(2\tau(1-\sigma),2\tau\sigma)$. 
This shows that $\{u_n\}_{n=1}^\infty$ is the Cauchy sequence in $X$ 
and then it converges to an element $y_\infty$ satisfying
$\|y_\infty\|_{X}=2^\alpha \sqrt{\Gamma(2\alpha)}\renorm{x}_{R(A^{\alpha})}$. 
Since $y_\infty$ has the spectral representation 
\begin{align*}
\lr{y_\infty,z}_X
&=
\frac{1}{\Gamma(\alpha)}
\int_0^\infty t^{\alpha}\left(\int_{\R}e^{-t\lambda}d\lr{E(\lambda)x,z}_X\right)\,\frac{dt}{t}
\\
&=
\int_{\R}\lambda^{-\alpha}d\lr{E(\lambda)x,z}_X
\\
&=\lr{A^{-\alpha}x,z}_X
\end{align*}
for all $z\in X$, where 
$E(\cdot)$ is 
the associated spectral measure of the selfadjoint operator $A$.
Therefore we conclude that 
$y_\infty\in D(A^\alpha)$ and $A^{\alpha}y_\infty=x\in R(A^\alpha)$.  
Conversely, let $x=A^\alpha y\in R(A^{\alpha})$. Then 
similarly we have 
\begin{align*}
\int_{0}^n
\|t^{\alpha}e^{-tA}x\|_X^2
\,\frac{dt}{t}
&=
\int_0^n t^{2\alpha}\left(\int_{\R}\lambda^{2\alpha}e^{-2t\lambda}d\lr{E(\lambda)y,y}_X\right)\,\frac{dt}{t}
\\
&=
\int_{\R}
\left(\int_0^n\lambda^{2\alpha}e^{-2t\lambda}\,\frac{dt}{t}
\right)d\lr{E(\lambda)y,y}_X
\\
&\leq 
4^{-\alpha}\Gamma(2\alpha)\|y\|_{X}^2.
\end{align*}
This means $\renorm{x}_{R(A^\alpha)}<+\infty$. 

A direct proof of the completeness is the following. 
Let $\{x_n\}_{n=1}^\infty$ be a Cauchy sequence in $R(A^{\alpha})$. 
By the completeness of $X$, 
there exists an element $x_{\infty}\in X$ 
such that
$x_n\to x_{\infty}$ in $X$ as $n\to \infty$. 
We prove $x_\infty\in R(A^{\alpha})$.
For arbitrary fixed $T>0$, we can see from 
the dominated convergence theorem that 
for $n\in\N$,
\begin{align*}
\int_{0}^T \|t^{\alpha}e^{-tA}(x_n-x)\|_X^2\,
\frac{dt}{t}
&=
\lim_{m\to \infty}\int_{0}^T \|t^{\alpha}e^{-tA}(x_n-x_m)\|_X^2\,\frac{dt}{t}
\\
&\leq 
\limsup_{m\to \infty}\renorm{x_n-x_m}_{R(A^{\alpha})}^2.
\end{align*}
Since the right-hand side of the above inequality is independent of $T$, 
the monotone convergence theorem 
provides that $x_n-x_\infty$ 
belongs to $R(A^{\alpha})$ (for all $n\in\N$), 
and so is $x_\infty$.
Consequently, we have $x_n\to x_\infty$ in $R(A^{\alpha})$ 
as $n\to \infty$. 
\end{proof}

We recall some of standard consequence of nonnegative selfadjoint operators. 
The following lemma explains 
that the condition $0\in \sigma(A)\setminus\sigma_{\rm p}(A)$ makes the situation simpler. 
\begin{lemma}\label{lem:fundamental}
The following assertions hold:
\begin{itemize}
\item[\bf (i)] 
For every $\alpha>0$, $R(A^{\alpha})$ is dense in $X$.
\item[\bf (ii)] 
For every $x\in X$, one has $\lim\limits_{t\to\infty}\|e^{-tS}x\|_X=0$.
\end{itemize}
\end{lemma}
\begin{proof}
{\bf (i)} It suffices to show that $R(A^n)$ is dense in $X$ for all $n\in\N$. 
We argue by induction. 
First we prove the case $n=1$. 
If $x\in X$ satisfies 
$\lr{x,Ay}_X=0$ for all $y\in D(A)$. 
Then by definition we have $x\in D(A^*)=D(A)$
with $Ax=A^*x=0$ and therefore $x=0$. This means that $R(A)$ is dense in $X$.
Next we assume the density of $R(A^k)$ in $X$ for some $k\in\N$. 
Then by the approximation argument via the resolvent, we see that 
$D(A)\cap R(A^k)$ is also dense in $D(A)$ in the sense of the graph norm.
If $x\in X$ satisfy $\lr{x,A^{k+1}y}_X=0$ for $y\in D(A^{k+1})$, then 
the functional $x_*:z\in D(A)\cap R(A^k)\mapsto \lr{x,Az}_X$ is densely defined in $D(A)$ 
which is identically zero. By density, we can uniquely extend the functional $x_*$ 
to 
the bounded functional $x_{**}:z\in D(A)\mapsto \lr{x,Az}_X$ 
which is also identically zero.  
The rest is the same as the case $n=1$. Consequently, we have 
$x=0$.

{\bf (ii)}
Let $x\in X$ be arbitrary fixed. 
Then for every $\ep>0$, 
by {\bf (ii)} with $\alpha=1$
there exists an element $y\in D(A)$ such that 
$\|x-Ay\|_H<\frac{\ep}{2}$.
The analyticity $\|e^{-tA}Ay\|\leq t^{-1}\|y\|_X$ 
and the contractivity 
give that for $t\geq 4\ep^{-1}\|y\|_{X}$, 
\begin{align*}
\|e^{-tA}x-x\|_X
&\leq \|e^{-tA}(x-Ay)\|_X+\|e^{-tA}Ay\|_X
\\
&\leq \|x-Ay\|_X+t^{-1}\|y\|_{X}
\\
&<\ep.
\end{align*}
The proof is complete.
\end{proof}

Thanks to Lemma \ref{lem:range:complete}, 
the operator $A$ can be successfully 
realized in the Hilbert space $R(A^\alpha)$. 
The definition is the following. 
\[
\begin{cases}
D(A_\alpha):=D(A)\cap R(A^{\alpha}),
\\
A_\alpha u=Au.
\end{cases}
\]
\begin{lemma}
For every $\alpha>0$, 
the operator $A_\alpha$ is 
densely defined, nonnegative and selfadjoint in $R(A^\alpha)$.
\end{lemma}
\begin{proof}
Since $A_\alpha$ is clearly symmetric and nonnegative in $R(A^\alpha)$, 
it suffices to prove only the range condition
$R(I+A_\alpha)=R(A^\alpha)$ (see Brezis \cite[Proposition 7.1]{Book_Brezis}). 
For $x\in R(A^\alpha)\subset X$, we put $y=(I+A)^{-1}x\in D(A)$. 
Then clearly we see that $y\in D(A)\cap R(A^\alpha)=D(A_\alpha)$
and $(1+A_\alpha)y=(1+A)y=x$. This provides $R(I+A_\alpha)=R(A^\alpha)$. 
\end{proof}

The following inequality is a kind of interpolation inequality 
which is essential to analyse wave equations with potentials in this paper.
\begin{lemma}\label{lem:embedding1}
If $\alpha\in (0,\frac{1}{2})$, then every $x\in D(A_{\alpha}^{1/2})=D(A^{1/2})\cap R(A^{\alpha})$, 
\[
\|x\|_X\leq 
2^{\frac{1}{2}+\alpha(1-2\alpha)}
\renorm{A^{1/2}x}_{R(A^{\alpha})}^{2\alpha}
\renorm{x}_{R(A^{\alpha})}^{1-2\alpha}.
\]
If $\alpha=\frac{1}{2}$, then 
the identity $\|x\|_X=2^{1/2}\renorm{A^{1/2}x}_{R(A^{1/2})}$ 
holds for all $x\in D(A^{1/2})$.
\end{lemma}
\begin{proof}
We start with the following identity:
\begin{align*}
\|x\|_X^2-\|e^{-tA}x\|_X^2
=2\int_0^t\|A^{1/2}e^{-\tau A}x\|_X^2\,d\tau.
\end{align*}
Recalling Lemma \ref{lem:fundamental}, we have the desired assertion for $\alpha=\frac{1}{2}$.
If $\alpha\in (0,\frac{1}{2})$, then 
we see from the H\"older inequality and 
the analyticity 
$\|A^{1/2}e^{-\frac{t}{2}A}x\|_{X}^2\leq \frac{2}{t}\|x\|_X^2$
that
\begin{align*}
\int_0^t\|A^{1/2}e^{-\tau A}x\|_X^2\,d\tau
&\leq 
\left(
\int_0^t \|\tau^{\alpha}e^{-\tau A}A^{1/2}x\|_X^2\,\frac{d\tau}{\tau}
\right)^{2\alpha}
\left(
\int_0^t\|\tau^{\alpha}A^{1/2}e^{-\tau A}x\|_X^2\,d\tau
\right)^{1-2\alpha}
\\
&\leq 
\renorm{A^{1/2}x}_{R(A^{\alpha})}^{4\alpha}
\left(
2\int_0^t \|\tau^{\alpha}e^{-\frac{\tau}{2}A}x\|_X^2\,\frac{d\tau}{\tau}
\right)^{1-2\alpha}.
\end{align*}
Letting $t\to\infty$, we obtain 
the desired inequality.
\end{proof}

Here we consider the Cauchy problem of the second-order evolution equations:
\begin{equation}\label{eq:abstract-wave}
\begin{cases}
w''(t)+Aw(t)=0, &\ t>0, 
\\
(w,w')(0)=(0,g),
\end{cases}
\end{equation}
where $g\in R(A^{\alpha})$ for some $\alpha>0$. 
Since the realization $A_\alpha$ is nonnegative and selfadjoint in $R(A^\alpha)$,
the problem \eqref{eq:abstract-wave} can be uniquely solved in the Hilbert space $R(A^\alpha)$. 
We denote $W_{A}(\cdot)g$ 
as the unique solution $w(\cdot)$ of \eqref{eq:abstract-wave}.
Then we further find that 
\begin{proposition}
\label{prop:abstractwave}
Let $\alpha\in (0,\frac{1}{2}]$ and $g\in R(A^{\alpha})$. Then 
$W_A(\cdot)g$ satisfies
\[
t^{2\alpha-1}\|W_A(\cdot)g\|_X\leq
2^{\frac{1}{2}+\alpha(1-2\alpha)}
\renorm{g}_{R(A^\alpha)}, \quad t>0. 
\]
In particular, if $\alpha=\frac{1}{2}$, 
then $W_A(\cdot)g$ is uniformly bounded in $X$. 
\end{proposition}
\begin{proof}
Since the energy conservation law in $R(A^{\alpha})$ holds, we find that 
\begin{align*}
\renorm{\tfrac{d}{dt}W_A(t)g}_{R(A^{\alpha})}^2
+
\renorm{A_\alpha^{1/2}W_A(t)g}_{R(A^{\alpha})}^2
=
\renorm{g}_{R(A^{\alpha})}^2, \quad t\geq 0
\end{align*}
and therefore $\renorm{W_A(t)g}_{R(A^{\alpha})}\leq  t\renorm{g}_{R(A^{\alpha})}$ for $t\geq0$. 
Using Lemma \ref{lem:embedding1}, we have
\begin{align*}
\|W_A(t)g\|_X
&\leq 
2^{\frac{1}{2}+\alpha(1-2\alpha)}
\renorm{A^{1/2}W_A(t)g}_{R(A^{\alpha})}^{2\alpha}
\renorm{W_A(t)g}_{R(A^{\alpha})}^{1-2\alpha}
\\
&\leq 
2^{\frac{1}{2}+\alpha(1-2\alpha)}
t^{1-2\alpha}
\renorm{g}_{R(A^{\alpha})}.
\end{align*}
The proof is complete. 
\end{proof}

To prove Theorem \ref{thm:opposite}, 
we need the following connection between two evolution operators $e^{-tA}$ and $W_A(t)$, 
which is motivated by the well-known Hadamard transmutation formula.
\begin{lemma}\label{lem:Hadamard}
For every $g\in X$, the following identity holds:
\[
e^{-tA}g=
\frac{1}{2}\pi^{-\frac{1}{2}}t^{-\frac{3}{2}}
\int_0^\infty \sigma e^{-\frac{\sigma^2}{4t}}
W_{A}(\sigma)g\,d\sigma, \quad t>0.
\]
\end{lemma}
\begin{proof}
It suffices to show it for $g\in D(A^{1/2})$. Put 
\begin{align*}
v(t)=
t^{-\frac{3}{2}}
\int_0^\infty \sigma e^{-\frac{\sigma^2}{4t}}
W_{A}(\sigma)g\,d\sigma
=
2t^{-\frac{1}{2}}
\int_0^\infty e^{-\frac{\sigma^2}{4t}}
W_A'(\sigma)g\,d\sigma, \quad t>0,
\end{align*}
where we can check that two different expression 
coincide by using integration by parts. 
By change of variables 
and the dominated convergence theorem, we have 
\[
v(t)=
2
\int_0^\infty \rho^{-\frac{1}{2}}e^{-\rho}
W_A'(\sqrt{4t\rho})g\,d\rho
\to 
2\pi^{\frac{1}{2}}g
\]
as $t\to 0$. 
On the other hand, by differentiating the representation used above, we also have 
\begin{align*}
v'(t)
&=
2t^{-\frac{1}{2}}\int_0^\infty e^{-\rho}
W_A''(\sqrt{4t\rho})g \, d\rho
\\
&=
2t^{-\frac{1}{2}}\int_0^\infty e^{-\rho}
\Big(-AW_A(\sqrt{4t\rho})g\Big)\,d\rho
\\
&=
-t^{-\frac{3}{2}} A\left(\int_0^\infty \sigma e^{-\frac{\sigma^2}{4t}}
W_A(\sigma)g\,d\sigma\right)
=-Av(t).
\end{align*}
The uniqueness of the Cauchy problem of $v'+Av=0$ shows that 
$v(t)=2\pi^{\frac{1}{2}}e^{-tA}g$.
\end{proof}

\section{Applications to wave equations with potentials}\label{sec:app}

Theorem \ref{thm:wave-bounded} 
is the direct consequence of Proposition \ref{prop:abstractwave}. 

\begin{proof}[Proof of Theorem \ref{thm:wave-bounded}]
Let $\alpha\in I_S\cap (0,\frac{1}{2}]$ and $g\in C_0^\infty(\Omega)$. Then 
we see by Theorem \ref{thm:equivalence} that $g\in R(S^{\alpha})$. 
Therefore Proposition \ref{prop:abstractwave} provides that 
\[
t^{2\alpha-1}\|W_S(t)g\|_{L^2(\Omega)}\leq
2^{\frac{1}{2}+\alpha(1-2\alpha)}
\renorm{g}_{R(S^\alpha)}, \quad t>0. 
\]
The proof is complete.
\end{proof}

Next, we give a proof of Theorem \ref{thm:opposite}. 
\begin{proof}[Proof of Theorem \ref{thm:opposite}]
Let $\alpha_*\in (0,\frac{1}{2}]$. 
Assume that for every $g\in C_0^\infty(\Omega)$, 
the solution $W_S(\cdot)g$ of \eqref{intro:eq:hyper}
satisfies 
\[
t^{2\alpha_*-1}\|W_S(\cdot)g\|_{L^2}\leq C_g, \quad t\geq 0
\]
for some positive constant $C_g$ (depending only on $g$).
Then using Lemma \ref{lem:Hadamard}, 
we can deduce that for every $t>0$, 
\begin{align*}
\|e^{-tS}g\|_X
&=
\frac{1}{2}\pi^{-\frac{1}{2}}t^{-\frac{3}{2}}
\left\|
\int_0^\infty \sigma e^{-\frac{\sigma^2}{4t}}
W_{S}(\sigma)g\,d\sigma
\right\|_X
\\
&\leq 
\frac{1}{2}\pi^{-\frac{1}{2}}t^{-\frac{3}{2}}
\int_0^\infty \sigma e^{-\frac{\sigma^2}{4t}}
\|W_{S}(\sigma)g\|_X \,d\sigma
\\
&\leq 
\frac{C_g}{2}\pi^{-\frac{1}{2}}
t^{-\frac{3}{2}}
\int_0^\infty \sigma^{2-2\alpha_*} e^{-\frac{\sigma^2}{4t}}
\,d\sigma
\\
&=
2^{1-2\alpha_*}\pi^{-\frac{1}{2}}C_g
\Gamma\left(\frac{3-2\alpha}{2}\right)t^{-\alpha_*}.
\end{align*}
Therefore we can conclude that $\renorm{g}_{R(S^\alpha)}<+\infty$ for all $\alpha\in (0,\alpha_*)$.
\end{proof}

To close the paper, we give a proof of Proposition \ref{prop:hardy-critical}.
\begin{proof}[Proof of Proposition \ref{prop:hardy-critical}]
Let $g\in C_0^\infty(\R^N\setminus\{0\})$ be a radially symmetric function 
satisfying \eqref{intro:prop:hardy-critical:ass}. 
Put $w(x,t)=[W_{S_{\lambda_*}}(t)g](x)$.
For each $t>0$, 
we also put a new two-dimensional radially symmetric function 
$y\in \R^2\setminus \{0\}\mapsto w_*(y,t)$ as   
\[
w_*(y,t)=|x|^{\frac{N-2}{2}}w(x,t), \quad |x|=|y|
\]
(which is used for the corresponding critical parabolic equation \eqref{intro:eq:para}
in Ioku--Ogawa \cite{IokuOgawa2019}). 
Then we can see that for every $t>0$ and $r=|y|=|x|$, 
\begin{align*}
\Delta_{y}w_*
&=\pa_r^2w_*+\frac{1}{r}\pa_rw_*
\\
&=r^{\frac{N-2}{2}}
\left(
\pa_r^2w+\frac{N-1}{r}\pa_r w+\Big(\frac{N-2}{2}\Big)^{2}\frac{w}{r^2}\right).
\\
&=-r^{\frac{N-2}{2}}S_{\lambda_*}w.
\end{align*}
Observe that the Friedrichs extension 
of the minimal operator for $-\Delta_{\R^2\setminus \{0\}}$ coincides 
with that of $-\Delta_{\R^2}$ (that it, the singular point $y=0$ is negligible). 
Therefore we can find that $w_*$ satisfies 
the two-dimensional wave equation in the whole space:
\begin{equation}\label{eq:2d-wave}
\begin{cases}
\pa_t^2w_*(y,t)-\Delta_y w_*(y,t)=0 &\text{in}\ \R^2\times (0,\infty),
\\
(w_*,\pa_tw_*)(x,0)=(0,g_*(y)) &\text{in}\ \R^2,
\end{cases}
\end{equation}
where $g_*$ is given by the same manner as $w_*$, that is, $g_*(y)=g(x)$ with $|x|=|y|$. 
Noting that 
\begin{align*}
\int_{\R^2}g_*(y)\,dy
=
\int_{\R^2}g(x)|x|^{-\frac{N-2}{2}}dx\neq0, 
\end{align*}
we can adopt the result in Ikehata \cite[Theorem 1.2]{Ikehata2023}. Namely, 
one has
\[
c\log t \leq \|w_*(t)\|_{L^2(\R^2)}^2\leq C\log t, 
\quad t\geq 2 
\]
for some positive constants $c$ and $C$. By the definition of $w_*$, we obtain
\[
c'\log t \leq \|w(t)\|_{L^2(\R^N)}^2\leq C'\log t, \quad t\geq 2
\]
for some positive constants $c'$ and $C'$. The proof is complete. 
\end{proof}


\end{document}